\documentclass[12pt]{amsproc}

\usepackage{setspace}
\usepackage{latexsym}
\usepackage{amsmath, amssymb, amsfonts, amsthm}
\usepackage{float}
\usepackage{color}
\usepackage{mathrsfs}
\usepackage{hyperref} 
\hypersetup{citecolor=red, linkcolor=blue, colorlinks=true}

\usepackage{bm}
\usepackage[margin=1in]{geometry}
\usepackage{url}

\newtheorem{thm}{Theorem}[section]
\newtheorem{lem}[thm]{Lemma}
\newtheorem{prop}[thm]{Proposition}
\newtheorem{cor}[thm]{Corollary}

\renewcommand\le{\leqslant}

\newcommand{\calP}{\mathcal{P}}

\newcommand{\calQ}{\mathcal{Q}}

\newcommand{\pg}{{\rm pg}}
\newcommand{\calS}{\mathcal{S}}
\newcommand{\Cay}{{\rm Cay}}

\title{Restrictions on parameters of partial difference sets in nonabelian groups}

\author{Eric Swartz}
\address{Department of Mathematics, William \& Mary, P.O. Box 8795, Williamsburg, VA 23187-8795, USA}
\email{easwartz@wm.edu}

\author{Gabrielle Tauscheck}
\address{Department of Mathematics, University of South Carolina, 1523 Greene Street, Columbia, SC 29208-4014, USA}
\email{tauscheg@email.sc.edu}

\begin{document}

\begin{abstract}
A partial difference set $S$ in a finite group $G$ satisfying $1 \notin S$ and $S = S^{-1}$ corresponds to an undirected strongly regular Cayley graph ${\rm Cay}(G,S)$.  While the case when $G$ is abelian has been thoroughly studied, there are comparatively few results when $G$ is nonabelian.  In this paper, we provide restrictions on the parameters of a partial difference set that apply to both abelian and nonabelian groups and are especially effective in groups with a nontrivial center.  In particular, these results apply to $p$-groups, and we are able to rule out the existence of partial difference sets in many instances.
\end{abstract}

\maketitle

\section{Introduction}
\label{sect:intro}

A subset $S$ of a finite group $G$ is a $(v, k, \lambda, \mu)$-\textit{partial difference set} (PDS) of $G$ if $|G| = v$, $|S| = k$, and each nonidentity element $g \in G$ can be written in either $\lambda$ or $\mu$ different ways (depending on whether or not $g$ is in S) as $g = ab^{-1}$, where $a,b \in S$, and, following \cite{M}, $S$ is said to be a \textit{regular PDS} if $1 \notin S$ and $S^{-1} = S$ (i.e., $s^{-1} \in S$ whenever $s \in S$).  

Part of the motivation to find examples of regular partial difference sets is the connection with strongly regular Cayley graphs.  Given a subset $S$ of the group $G$, the \textit{Cayley graph} $\Cay(G,S)$ is defined to be the graph with vertex set the elements of $G$ such that $g, h \in G$ are adjacent if and only if $gh^{-1} \in S$.  If the set $S$ is a $(v, k, \lambda, \mu)$-PDS, then the Cayley graph $\Cay(G,S)$ is a $(v, k, \lambda, \mu)$-\textit{strongly regular graph} (SRG) \cite[Proposition 1.1]{M}, which means that $\Cay(G,S)$ has $v$ vertices, $\Cay(G,S)$ is regular of degree $k$, any two adjacent vertices in $\Cay(G,S)$ have exactly $\lambda$ common neighbors, and any two nonadjacent vertices in $\Cay(G,S)$ have exactly $\mu$ common neighbors.  For further applications of partial difference sets to coding theory and finite geometry, see the survey of Ma \cite{M}.

The case when $G$ is abelian has been thoroughly studied; see \cite{M} for a survey of older results and \cite{DKW, DNW, DW1, DW2, FMX, MS, MX, Polhill, W1, W2} for a number of very recent results.  On the other hand, comparatively little is known in the case when $G$ is nonabelian.  There have been constructions in sporadic cases (see, for instance, \cite{GMV, JK}) and some instances of constructions of infinite families (see \cite{FHC, FL,  Swartz}).  At the same time, there have been relatively few results dealing with the nonabelian case in general.  For instance, in the abelian case, there are restrictions such on \textit{(numerical) multipliers} of the set $S$ (see \cite[Theorem 4.1]{M}), whereas in the general case when the group is allowed to be nonabelian, the corresponding restriction is weaker (see \cite[Theorem 4.3]{M}, originally published in \cite{GL}).  

The purpose of this paper is to provide new restrictions on partial difference sets that apply to nonabelian groups as well as abelian groups.  In \cite{Yoshiara}, Yoshiara was able to provide restrictions on groups acting regularly on the point set of a finite \textit{generalized quadrangle} (see Subsection \ref{sub:pggq}).  Since the \textit{collinearity graph} of a generalized quadrangle is a strongly regular graph, this provides restrictions on the strongly regular Cayley graphs (and, hence, partial difference sets) in this specific instance.  One of the main ingredients in Yoshiara's analysis was Benson's Lemma (see Lemma \ref{lem:Benson}), which has since been generalized to strongly regular graphs by DeWinter, Kamischke, and Wang \cite{DKW}.  The main idea of this paper is to generalize the results of Yoshiara to general strongly regular graphs using this recent result of DeWinter, Kamischke, and Wang.

This paper is organized as follows.  In Section \ref{sect:prelim}, we introduce the preliminary combinatorial results about strongly regular graphs and summarize Yoshiara's results from \cite{Yoshiara}.  In Section \ref{sect:YoshSRG}, we generalize Yoshiara's results to strongly regular graphs.  We apply these results to \textit{partial geometries} in Section \ref{sect:pg}, and we produce simple numerical conditions in Section \ref{sect:rest} that rule out the existence of partial difference sets in many instances, especially when the group has a nontrivial center.  Finally, in Section \ref{sect:infeas} we include a table of parameters of partial difference sets that are ruled out in groups having a nontrivial center using the results in Section \ref{sect:rest}.

\section{Preliminaries}
\label{sect:prelim}

\subsection{Strongly regular graphs and their parameters}

Let $\Gamma$ be a $(v,k,\lambda, \mu)$-strongly regular graph.  If $\Gamma$ has adjacency matrix $A$, then $A$ has eigenvalues $\nu_1 = k$,
\[\nu_2 = \frac{1}{2}\left(\lambda - \mu + \sqrt{\Lambda} \right), \quad \nu_3 = \frac{1}{2}\left(\lambda - \mu - \sqrt{\Lambda} \right),\]
where $\Lambda = (\lambda - \mu)^2 + 4(k - \mu) = (\nu_2 - \nu_3)^2$.  A \textit{conference graph} is a strongly regular graph with $k = (v-1)/2$, $\lambda = (v-5)/4$, and $\mu = (v-1)/4$.  If the strongly regular graph $\Gamma$ is not a conference graph, then $2k + (v-1)(\lambda - \mu) \neq 0$, and the eigenvalues $\nu_2$ and $\nu_3$ are integers.

The following lemma collects some well-known results about the relationships among the parameters of a strongly regular graph.  (See, for instance, \cite{BCN} or \cite{GodsilRoyle}.)  These results are used freely later in the paper.

\begin{lem}
 \label{lem:SRGparameters}
 Let $\Gamma$ be a $(v,k,\lambda, \mu)$-strongly regular graph with eigenvalues $k, \nu_2, \nu_3$, where $\nu_2 > \nu_3$.  Then, the following hold:
 \begin{itemize}
  \item[(1)] $k(k - \lambda - 1) = (v - k - 1)\mu$; 
  \item[(2)] $k = \mu - \nu_2 \nu_3$;
  \item[(3)] $\mu v = (k - \nu_2)(k - \nu_3)$, and, if $\mu > 0$, then $v = (k - \nu_2)(k - \nu_3)/\mu$;
  \item[(4)] $\lambda = \mu + \nu_2 + \nu_3$;
  \item[(5)] the complement of $\Gamma$ is a $(v, v-k-1, v -2k + \mu -2, v - 2k + \lambda)$-strongly regular graph with eigenvalues $-\nu_3 - 1$ and $-\nu_2 - 1$.
 \end{itemize}
\end{lem}

\subsection{Partial geometries and generalized quadrangles}
\label{sub:pggq}

A \textit{partial geometry} $\pg(s,t, \alpha)$ is an incidence geometry of points and lines satisfying the following three conditions:

\begin{itemize}
 \item[(i)] each point is incident with $t + 1$ lines, and two points are mutually incident with at most one line;
 \item[(ii)] each line is incident with $s + 1$ points, and two lines are mutually incident with at most one point;
 \item[(iii)] if $P$ is a point and $\ell$ is a line not incident with $P$, there are exactly $\alpha$ points on $\ell$ collinear with $P$.
\end{itemize}

If $\calS$ is a $\pg(s,t, \alpha)$, then we say it has order $(s,t, \alpha)$.  If $\alpha = 1$, then the partial geometry is a \textit{generalized quadrangle} of order $(s,t)$.

Given a partial geometry $\calS$ of order $(s,t, \alpha)$, one may define the associated \textit{collinearity graph} $\Gamma_{\calS}$, which has vertex set the points of $\calS$ and edge set consisting of pairs of collinear points.  It is well known that the collinearity graph of a $\pg(s,t,\alpha)$ is an $((s+1)(st + \alpha)/\alpha, s(t+1), s - 1 + t(\alpha - 1), \alpha(t+1))$-strongly regular graph.

\subsection{Automorphisms of generalized quadrangles and strongly regular graphs}

In this subsection, we present some of the main ideas of Yoshiara in \cite{Yoshiara} and the recent result of DeWinter, Kamischke, and Wang \cite{DKW}, which is a generalization of Benson's Lemma to arbitrary strongly regular graphs.  Throughout this paper, if $G$ is a group and $x \in G$, then $x^G$ denotes the conjugacy class of $G$ containing $x$ and $C_G(x)$ denotes the centralizer of $x$ in $G$.

\begin{lem}\cite[Lemma 3]{Yoshiara}
\label{lem:Yoshiara1}
 Suppose the group $G$ acts regularly on the point set $\calP$ of a generalized quadrangle $\calQ$ of order $(s,t)$.  Fix a distinguished point $O$ of $\calQ$.  Let $a \in G$ be a nontrivial automorphism, $\Delta := \{g \in G: O^g \sim O\} \cup \{1\}$, $\Delta^c$ be the set-theoretic complement of $\Delta$ in $G$, and let $d_1(a)$ be the number of points that mapped to collinear points but are not fixed by $a$.  Then,
 \[d_1(a) = |a^G \cap \Delta| |C_G(a)| = (s+1)(t+1) + (s+t)u_a \]
 for some integer $u_a$,  Furthermore,
 \[ |a^G \cap \Delta^c| |C_G(a)|= t(s-1)(s+1) - (s+t)u_a.\]
\end{lem}

\begin{lem}\cite[Lemma 6]{Yoshiara}
\label{lem:Yoshiara2}
 Suppose the group $G$ acts regularly on the point set $\calP$ of a generalized quadrangle $\calQ$ of order $(s,t)$, let $r = \gcd(s,t)$, and let $a$ be a nontrivial element of $G$.  For a distinguished point $O$ of $\calQ$, define $\Delta := \{ g \in G: O^g \sim O\} \cup \{1\}$.  Then, the following hold.
 \begin{itemize}
  \item[(1)] If $r > 1$, then $a^G \cap \Delta \neq \varnothing$.
  \item[(2)] $|a^G \cap \Delta^c|$ is a multiple of $r$ (possibly equal to $0$).
 \end{itemize}
\end{lem}

Lemma \ref{lem:Yoshiara2} follows almost immediately from Lemma \ref{lem:Yoshiara1}.  One of the main ingredients in the proof of Lemma \ref{lem:Yoshiara1} is Benson's Lemma.

\begin{lem}\cite[Lemma 4.3]{Benson}
\label{lem:Benson}
 If $x$ is an automorphism of a finite generalized quadrangle of order $(s,t)$, $d_0(x)$ denotes the number of fixed points of $x$, and $d_1(x)$ denotes the number of points that are sent to collinear points by $x$ (but are not fixed), then
 \[(t+1)d_0(x) + d_1(x) \equiv (st+1) \pmod {s+t}.\]
\end{lem}

Benson's Lemma is a specific instance of a technique, attributed to Graham Higman, which calculates the value of a character of the automorphism group of an association scheme on an eigenspace; see \cite[pp. 89--91]{CameronPerm}.  In particular, a generalization of Benson's Lemma has been proved by De Winter, Kamischke, and Wang for non-conference strongly regular graphs.

\begin{lem}\cite[Theorem 1]{DKW}
\label{lem:DKW}
Let $\Gamma$ be a $(v,k, \lambda, \mu)$-strongly regular graph whose adjacency matrix has integer eigenvalues $k, \nu_2$, and $\nu_3$, where $\nu_2 > \nu_3$.  If $x$ is a nontrivial automorphism of $\Gamma$ that such that $x$ fixes $d_0(x)$ vertices of $\Gamma$ and sends $d_1(x)$ vertices to adjacent vertices, then
\[k - \nu_3 \equiv -\nu_3 d_0(x)  + d_1(x) \pmod {(\nu_2 - \nu_3)}. \]
\end{lem}

\section{Generalizing Yoshiara's results to strongly regular graphs}
\label{sect:YoshSRG}

Throughout this section, we will assume that $\Gamma$ is a $(v,k,\lambda, \mu)$-strongly regular graph with eigenvalues $k$, $\nu_2$, and $\nu_3$, with $\nu_2 > \nu_3$.  We will assume that the group $G$ acts regularly on the vertex set of $\Gamma$, and, for a fixed vertex $\alpha$ of $\Gamma$, we let $\Delta = \{g \in G: \alpha^g \sim \alpha\} \cup \{1\}$.  For a nontrivial element $x \in G$, we define $d_1(x)$ to be the number of vertices sent to adjacent vertices by $x$. 

By replacing Benson's Lemma (Lemma \ref{lem:Benson}) with the generalization by De Winter, Kamischke, and Wang (Lemma \ref{lem:DKW}), we are able to prove a generalization of Lemma \ref{lem:Yoshiara1} for non-conference strongly regular graphs.  

\begin{lem}
 \label{lem:main1}
If $x$ is a nontrivial element of $G$, then 
\[d_1(x) = k - \nu_3 + u_x(\nu_2 - \nu_3) = \mu - \nu_3(\nu_2 + 1) + u_x(\nu_2 - \nu_3) = |x^G \cap \Delta| |C_G(x)|\]
for some integer $u_x$.  Furthermore, assuming $\mu > 0$, we have
\[|x^G \cap \Delta^c||C_G(x)| = \frac{\nu_2\nu_3(\nu_2 + 1)(\nu_3 + 1)}{\mu} - \nu_2(\nu_3 + 1) - u_x(\nu_2 - \nu_3).\]
\end{lem}

\begin{proof}
 Let $\Gamma$ be such a $(v,k,\lambda, \mu)$-strongly regular graph with regular group of automorphisms $G$, and let $x$ be a nontrivial element of $G$.  Since $x$ does not fix any vertices, it follows by Lemma \ref{lem:DKW} that, for some integer $u_x$, 
 \[ d_1(x) = k - \nu_3 + u_x(\nu_2 - \nu_3).\]  Equivalently, using Lemma \ref{lem:SRGparameters} (2), we see that
 \[ d_1(x) = \mu - \nu_3(\nu_2 + 1) + u_x(\nu_2 - \nu_3).\]
 Now, since $G$ acts regularly on the vertices of $\Gamma$, for a fixed vertex $\alpha$, we have $V(\Gamma) = \{ \alpha^g : g \in G\}.$  Suppose $\alpha^g$ is sent to an adjacent vertex by $x$, i.e., suppose $\alpha^{gx} \sim \alpha^g$.  This happens if and only if $\alpha^{gxg^{-1}} \sim \alpha$, which implies that $gxg^{-1} \in x^G \cap \Delta$.  Since $gxg^{-1} = hxh^{-1}$, where $g,h \in G$, if and only if $g^{-1}h \in C_G(x)$, we see that $d_1(x) = |x^G \cap \Delta||C_G(x)|$.  Finally, since $|x^G \cap \Delta^c| = |x^G| - |x^G \cap \Delta|$ and 
 \[|x^G||C_G(x)| = v = \frac{(k - \nu_2)(k-\nu_3)}{\mu} = \frac{(\mu - \nu_2 \nu_3 - \nu_3)(\mu - \nu_2\nu_3 - \nu_3)}{\mu}, \]
 we have 
 \[ |x^G \cap \Delta^c||C_G(x)| = \frac{\nu_2\nu_3(\nu_2 + 1)(\nu_3 + 1)}{\mu} - \nu_2(\nu_3 + 1) - u_x(\nu_2 - \nu_3),\]
 as desired.
\end{proof}

\begin{lem}
 \label{lem:main2}
 If $(\nu_2 - \nu_3)$ does not divide $\mu - \nu_3(\nu_2 + 1)$ and $x$ is a nontrivial element of $G$, then $x^G \cap \Delta \neq \varnothing$.
\end{lem}

\begin{proof}
 If $x^G \cap \Delta = \varnothing$, then, by Lemma \ref{lem:main1}, 
 \[\mu - \nu_3(\nu_2 + 1) + u_x(\nu_2 - \nu_3) = 0. \]
 The result follows.
\end{proof}

\begin{lem}
 \label{lem:main3}
 Let $x$ be any nontrivial element of $G$, let $r = \gcd(\nu_2(\nu_3 + 1), \nu_3(\nu_2 + 1))$, and assume $\gcd(r, \mu) = 1$.  Then, the following hold.
 \begin{itemize}
  \item[(1)] If $r > 1$, then $x^G \cap \Delta \neq \varnothing$.
  \item[(2)] $|x^G \cap \Delta^c|$ is a multiple of $r$ (possibly equal to $0$).
 \end{itemize}
\end{lem}

\begin{proof}
 We will prove (1) first.  If $x^G \cap \Delta = \varnothing$, then, by Lemma \ref{lem:main1}, it follows that
 \[ 0 = |x^G \cap \Delta| |C_G(x)| = \mu - \nu_3(\nu_2 + 1) + u_x(\nu_2 - \nu_3)\]
 for some integer $u_x$.  Noting that $\nu_2 - \nu_3 = \nu_2(\nu_3 + 1) - \nu_3(\nu_2 + 1)$, we have $r$ divides $\mu$.  However, by assumption, $\mu$ is coprime to $r$ and $r > 1$, a contradiction.  Hence, $x^G \cap \Delta \neq \varnothing$.
 
 To prove (2), we note first that, since $\gcd(r, \mu) = 1$, $r$ divides $\nu_2 \nu_3 (\nu_2 + 1)(\nu_3 + 1)/\mu$.  Since $r$ divides both $\nu_2(\nu_3 + 1)$ and $\nu_2 - \nu_3$, it follows from Lemma \ref{lem:main1} that $|x^G \cap \Delta^c| |C_G(x)|$ is a multiple of $r$.  On the other hand,
 \[|G| = v = \mu - \nu_2(\nu_3 + 1) - \nu_3(\nu_2 + 1) + \frac{\nu_2 \nu_3 (\nu_2 + 1)(\nu_3 + 1)}{\mu}. \]
 Since $r$ is coprime to $\mu$ but divides each other term, $r$ is coprime to $|G|$ and hence also to $|C_G(x)|$, implying that $r$ divides $|x^G \cap \Delta^c|$, as desired.
\end{proof}

\section{An application to partial geometries}
\label{sect:pg}

Groups acting on partial geometries have been studied before, see \cite{DeWinter, LMS, TTV}.  In particular, an analogue of Benson's Lemma was proved for partial geometries in \cite{DeWinter}, and the following result was obtained for abelian groups acting regularly on the point set of a partial geometry.

\begin{lem}\cite[Corollary 2.3]{DeWinter}
 If $\calS$ is a partial geometry $\pg(s,t, \alpha)$, $\alpha \neq s+1$, and $\calS$ admits an abelian point-regular group of automorphisms, then
 \[(s+1)\frac{st+\alpha}{\alpha} \equiv (s+1)(t+1) \equiv 0 \pmod {s+t - \alpha + 1}. \]
\end{lem}

In \cite{TTV}, the following result is proved.

\begin{lem}\cite[Corollary 4]{TTV}
 Let $\calS$ be a partial geometry of order $(s, t, \alpha)$, and let $x$ be an automorphism of $\calS$.  If $s$, $t$, and $\alpha - 1$ have a common divisor distinct from $1$, then there exists at least one fixed point or at least one point which is mapped to a point collinear to itself.
\end{lem}

We are able to apply Lemmas \ref{lem:main1}, \ref{lem:main2}, and \ref{lem:main3} to obtain similar results.  Note that $d_1$ and $\Delta$ are as defined in Section \ref{sect:YoshSRG}.

\begin{prop}
 \label{prop:pg}
 Let $G$ act regularly on the point set of a partial geometry $\calS$ of order $(s,t,\alpha)$.  If $x$ is a nontrivial element of $G$, then the following hold.
 
\begin{itemize}
 \item[(1)] \[d_1(x) = (s+1)(t+1) + u_x(s + t - \alpha + 1) = |x^G \cap \Delta||C_G(x)|,\]
 where $u_x$ is an integer, and, furthermore,
 \[|x^G \cap \Delta^c||C_G(x)| = \frac{t(s+1)(s - \alpha)}{\alpha} - u_x(s + t - \alpha + 1). \]
 \item[(2)] If $s + t - \alpha + 1$ does not divide $(s+1)(t+1)$, then $x^G \cap \Delta \neq \varnothing$.
 \item[(3)] Let $r = \gcd(t(s - \alpha), s + t - \alpha + 1)$, and suppose $r > 1$ and coprime with $\alpha(t+1)$.  Then $x^G \cap \Delta \neq \varnothing$ and $|x^G \cap \Delta^c|$ is a multiple of $r$ (possibly equal to $0$).
\end{itemize}
\end{prop}

\begin{proof}
 Noting that the collinearity graph of $\calS$ is an $((s+1)(st + \alpha)/\alpha, s(t+1), s - 1 + t(\alpha - 1), \alpha(t+1))$-strongly regular graph, the results for (1), (2), and (3) follow immediately from Lemmas \ref{lem:main1}, \ref{lem:main2}, and \ref{lem:main3}, respectively.
\end{proof}

\section{New restrictions on parameters of partial difference sets}
\label{sect:rest}

While partial geometries and generalized quadrangles do not have natural complements, strongly regular graphs do.  Moreover, since a group of automorphisms preserves both edges and non-edges, a group acting regularly on a strongly regular graph will also act regularly on its complement.  Immediately, this yields the following result.  We keep the notation from Section \ref{sect:YoshSRG} for the following results.

\begin{prop}
 \label{prop:comp}
 Let $x$ be a nontrivial element of $G$, which acts regularly on the non-conference $(v,k, \lambda, \mu)$-strongly regular graph $\Gamma$.  If $\nu_2 - \nu_3$ does not divide either of $\mu - \nu_3(\nu_2 + 1)$ or $v - 2k + \lambda - \nu_3(\nu_2 + 1)$, then $x^G \cap \Delta \neq \varnothing$ and $x^G \cap \Delta^c \neq \varnothing$.   
\end{prop}

\begin{proof}
 This follows immediately by applying Lemma \ref{lem:main2} both to $\Gamma$ and its complement. 
\end{proof}

\begin{cor}
 \label{cor:center}
 If $G$ is a group of order $v$ with a nontrivial center such that $\nu_2 - \nu_3$ divides neither $\mu - \nu_3(\nu_2 + 1)$ nor $v - 2k + \lambda - \nu_3( \nu_2 + 1)$, then $G$ cannot have a $(v,k,\lambda, \mu)$-PDS, i.e., $G$ cannot act regularly on the vertices of a non-conference $(v,k,\lambda, \mu)$-strongly regular graph.
\end{cor}

\begin{proof}
 If $z$ is a nontrivial element of $Z(G)$, then $z^G = \{z\}$.  If $G$ acts regularly on the vertices of a non-conference $(v,k,\lambda, \mu)$-strongly regular graph, by Proposition \ref{prop:comp}, $z^G \cap \Delta \neq \varnothing$ and $z^G \cap \Delta^c \neq \varnothing$, which is impossible.  
\end{proof}

\begin{cor}
 \label{cor:pgroup}
 Let $p$ be a prime and $G$ be a nontrivial $p$-group.  If $\nu_2 - \nu_3$ divides neither $\mu - \nu_3(\nu_2 + 1)$ nor $v - 2k + \lambda - \nu_3(\nu_2 + 1)$, then there does not exist a $(v,k,\lambda, \mu)$-PDS in $G$, i.e., $G$ cannot act regularly on the vertices of a non-conference $(v,k,\lambda, \mu)$-strongly regular graph.
\end{cor}

\begin{proof}
 This follows from Corollary \ref{cor:center} and the fact that nontrivial $p$-groups have nontrivial centers.
\end{proof}

We note that the conditions of Corollaries \ref{cor:center} and \ref{cor:pgroup} are nontrivial, since there are many feasible parameter sets to which these conditions apply; see Section \ref{sect:infeas}.

\section{Infeasible parameters of partial difference sets in groups with a nontrivial center}
\label{sect:infeas}

Tables \ref{tbl:nonp} and \ref{tbl:p} provide instances when partial difference sets can be ruled out entirely, since in every instance a group of order $v$ must have a nontrivial center.  Note that $\nu_2 - \nu_3 = \sqrt{\Lambda}$, and the column ``$m$'' represents the value of $\mu - \nu_3( \nu_2 + 1)$.  Note that calculating $v - 2k + \lambda - \nu_3(\nu_2 + 1)$ for a $(v,k,\lambda, \mu)$-SRG $\Gamma$ is equivalent to calculating $m = \mu - \nu_3( \nu_2 + 1)$ for the complement of $\Gamma$.

\begin{table}[H]\centering
\begin{minipage}{0.4\textwidth}\centering
 {\footnotesize \begin{tabular}{cccccccc}
  $v$ & $k$ & $\lambda$ & $\mu$ & $\nu_2$ & $\nu_3$ & $\sqrt{\Lambda}$ & $m$ \\
  \hline
  28  & 12 & 6 & 4 & 4 & -2 & 6 & 14 \\
      & 15 & 6 & 10 & 1 & -5 & 6 & 20 \\
  63  & 30 & 13 & 15 & 3 & -5 & 8 & 35 \\
      & 32 & 16 & 16 & 4 & -4 & 8 & 36 \\
  88  & 27 & 6 & 9 & 3 & -6 & 9 & 33 \\
      & 60 & 41 & 40 & 5 & -4 & 9 & 64 \\
  105 & 26 & 13 & 4 & 11 & -2 & 13 & 28 \\
      & 78 & 55 & 66 & 1 & -12 & 13 & 90 \\
  105 & 32 & 4 & 12 & 2 & -10 & 12 & 42 \\
      & 72 & 51 & 45 & 9 & -3 & 12 & 75 \\
  105 & 52 & 21 & 30 & 2 & -11 & 13 & 63 \\
      & 52 & 29 & 22 & 10 & -3 & 13 & 55 \\ 
  117 & 36 & 15 & 9 & 9 & -3 & 12 & 39 \\
      & 80 & 52 & 60 & 2 & -10 & 12 & 90 \\
  176 & 25 & 0 & 4 & 3 & -7 & 10 & 32 \\
      & 150 & 128 & 126 & 6 & -4 & 10 & 154 \\
  176 & 45 & 18 & 9 & 12 & -3 & 15 & 48 \\
      & 130 & 93 & 104 & 2 & -13 & 15 & 143 \\
  176 & 70 & 18 & 34 & 2 & -18 & 20 & 88 \\
      & 105 & 68 & 54 & 17 & -3 & 20 & 108 \\
  176 & 70 & 24 & 30 & 4 & -10 & 14 & 80 \\
      & 105 & 64 & 60 & 9 & -5 & 14 & 110 \\
  189 & 48 & 12 & 12 & 6 & -6 & 12 & 54 \\
      & 140 & 103 & 105 & 5 & -7 & 12 & 147 \\
  195 & 96 & 46 & 48 & 6 & -8 & 14 & 104 \\
      & 98 & 49 & 49 & 7 & -7 & 14 & 105 \\  
  208 & 75 & 30 & 25 & 10 & -5 & 15 & 80 \\
      & 132 & 81 & 88 & 4 & -11 & 15 & 143 \\
  \end{tabular}}
\end{minipage}\qquad
\begin{minipage}{0.4\textwidth}\centering
 {\footnotesize \begin{tabular}{cccccccc}
  $v$ & $k$ & $\lambda$ & $\mu$ & $\nu_2$ & $\nu_3$ & $\sqrt{\Lambda}$ & $m$ \\
  \hline
  208 & 81 & 24 & 36 & 3 & -15 & 18 & 88 \\
      & 126 & 80 & 70 & 14 & -4 & 18 & 130 \\
  225 & 96 & 51 & 33 & 21 & -3 & 24 & 99 \\
      & 128 & 64 & 84 & 2 & -22 & 24 & 150 \\
  231 & 30 & 9 & 3 & 9 & -3 & 12 & 33 \\
      & 200 & 172 & 180 & 2 & -10 & 12 & 210 \\
  231 & 40 & 20 & 4 & 18 & -2 & 20 & 42 \\
      & 190 & 153 & 171 & 1 & -19 & 20 & 209 \\
  231 & 90 & 33 & 36 & 6 & -9 & 15 & 99 \\
      & 140 & 85 & 84 & 8 & -7 & 15 & 147 \\
  232 & 33 & 2 & 5 & 4 & -7 & 11 & 40 \\
      & 198 & 169 & 168 & 6 & -5 & 11 & 203 \\
  232 & 63 & 14 & 18 & 5 & -9 & 14 & 72 \\
      & 168 & 122 & 120 & 8 & -6 & 14 & 174 \\
  232 & 77 & 36 & 20 & 19 & -3 & 22 & 80 \\
      & 154 & 96 & 114 & 2 & -20 & 22 & 174 \\
   232 & 81 & 30 & 27 & 9 & -6 & 15 & 87 \\
       & 150 & 95 & 100 & 5 & -10 & 15 & 160 \\
   236 & 55 & 18 & 11 & 11 & -4 & 15 & 59 \\
       & 180 & 135 & 144 & 3 & -12 & 15 & 192 \\ 
   275 & 112 & 30 & 56 & 2 & -28 & 30 & 140 \\
       & 162 & 105 & 81 & 27 & -3 & 30 & 165 \\
   279 & 128 & 52 & 64 & 4 & -16 & 20 & 144 \\
       & 150 & 85 & 75 & 15 & -5 & 20 & 155 \\
   285 & 64 & 8 & 16 & 4 & -12 & 17 & 76 \\
       & 220 & 171 & 165 & 11 & -5 & 16 & 225 \\
   297 & 128 & 64 & 48 & 20 & -4 & 24 & 132 \\
       & 168 & 87 & 105 & 3 & -21 & 24 & 189 \\
  \end{tabular}}
\end{minipage}
\caption{Parameters with $v \le 300$ ruled out by Corollary \ref{cor:center}}
\label{tbl:nonp}
\end{table}

\begin{table}[H]
 {\footnotesize \begin{tabular}{cccccccc}
  $v$ & $k$ & $\lambda$ & $\mu$ & $\nu_2$ & $\nu_3$ & $\sqrt{\Lambda}$ & $m$ \\
  \hline
  343 & 102 & 21 & 34 & 4 & -17 & 21 & 119 \\
      & 240 & 171 & 160 & 16 & -5 & 21 & 245 \\
  343 & 114 & 45 & 34 & 16 & -5 & 21 & 119 \\
      & 228 & 147 & 160 & 4 & -17 & 21 & 245 \\
  625 & 246 & 119 & 82 & 41 & -4 & 45 & 250 \\
      & 378 & 213 & 252 & 3 & -42 & 45 & 420 \\
  729 & 208 & 37 & 68 & 4 & -35 & 39 & 243 \\
      & 520 & 379 & 350 & 34 & -5 & 29 & 525 \\
  729 & 248 & 67 & 93 & 5 & -31 & 36 & 279 \\
      & 480 & 324 & 300 & 30 & -6 & 36 & 486 \\
  729 & 280 & 127 & 95 & 37 & -5 & 42 & 285 \\
      & 448 & 262 & 296 & 4 & -38 & 42 & 486 \\
  \end{tabular}}
 \caption{Parameters with $v \le 1000$ ruled out by Corollary \ref{cor:pgroup}}
 \label{tbl:p}
\end{table}

Tables \ref{tbl:center1} and \ref{tbl:center2} provide instances when partial difference sets can be ruled out for groups of order $v$ with nontrivial centers.  Note that $\nu_2 - \nu_3 = \sqrt{\Lambda}$, and the column ``$m$'' represents the value of $\mu - \nu_3( \nu_2 + 1)$. 

\begin{table}[H]
\begin{minipage}{0.4\textwidth}\centering
 {\footnotesize \begin{tabular}{cccccccc}
  $v$ & $k$ & $\lambda$ & $\mu$ & $\nu_2$ & $\nu_3$ & $\sqrt{\Lambda}$ & $m$ \\
  \hline
  36  & 14 & 7 & 4 & 5 & -2 & 7 & 16  \\
      & 21 & 10 & 15 & 1 & -6 & 7 & 27 \\
  50  & 21 & 8 & 9 & 3 & -4 & 7 & 25 \\
      & 28 & 15 & 16 & 3 & -4 & 7 & 32 \\
  66  & 20 & 10 & 4 & 8 & -2 & 10 & 22 \\
      & 45 & 28 & 36 & 1 & -9 & 10 & 54 \\ 
  70  & 27 & 12 & 9 & 6 & -3 & 9 & 30 \\
      & 42 & 23 & 28 & 2 & -7 & 9 & 49 \\
  78  & 22 & 11 & 4 & 9 & -2 & 11 & 24 \\
      & 55 & 36 & 45 & 1 & -10 & 11 & 65 \\
  96  & 35 & 10 & 14 & 3 & -7 & 10 & 42 \\
      & 60 & 38 & 36 & 6 & -4 & 10 & 64 \\
  100 & 33 & 14 & 9 & 8 & -3 & 11 & 36 \\
      & 66 & 41 & 48 & 2 & -9 & 11 & 75 \\
  120 & 28 & 14 & 4 & 12 & -2 & 14 & 30 \\
      & 91 & 66 & 78 & 1 & -13 & 14 & 104 \\
  120 & 42 & 8 & 18 & 2 & -12 & 14 & 54 \\
      & 77 & 52 & 44 & 11 & -3 & 14 & 80 \\
  126 & 25 & 8 & 4 & 7 & -3 & 10 & 28 \\
      & 100 & 78 & 84 & 2 & -8 & 10 & 108 \\
  126 & 50 & 13 & 24 & 2 & -13 & 15 & 63 \\
      & 75 & 48 & 39 & 12 & -3 & 15 & 78 \\
  126 & 60 & 33 & 24 & 12 & -3 & 15 & 63 \\
      & 65 & 28 & 39 & 2 & -13 & 15 & 78 \\
  130 & 48 & 20 & 16 & 8 & -4 & 12 & 52 \\
      & 81 & 48 & 54 & 3 & -9 & 12 & 90 \\
  136 & 30 & 8 & 6 & 6 & -4 & 10 & 34 \\
      & 105 & 80 & 84 & 3 & -7 & 10 & 112 \\
  136 & 60 & 24 & 28 & 4 & -8 & 12 & 68 \\
      & 75 & 42 & 40 & 7 & -5 & 12 & 80 \\
  136 & 63 & 30 & 28 & 7 & -5 & 12 & 68 \\
      & 72 & 36 & 40 & 4 & -8 & 12 & 80 \\
  148 & 63 & 22 & 30 & 3 & -11 & 14 & 74 \\
      & 84 & 50 & 44 & 10 & -4 & 14 & 88 \\
  148 & 70 & 36 & 30 & 10 & -4 & 14 & 74 \\
      & 77 & 36 & 44 & 3 & -11 & 14 & 88 \\
  154 & 48 & 12 & 16 & 4 & -8 & 12 & 56 \\
      & 105 & 72 & 70 & 7 & -5 & 12 & 110 \\
  154 & 72 & 26 & 40 & 2 & -16 & 18 & 88 \\
      & 81 & 48 & 36 & 15 & -3 & 18 & 84 \\
  170 & 78 & 35 & 36 & 6 & -7 & 13 & 85 \\
      & 91 & 48 & 49 & 6 & -7 & 13 & 98 \\
  171 & 34 & 17 & 4 & 15 & -2 & 17 & 36 \\
      & 136 & 105 & 120 & 1 & -16 & 17 & 152 \\
  171 & 50 & 13 & 15 & 5 & -7 & 12 & 57 \\
      & 120 & 84 & 84 & 6 & -6 & 12 & 126 \\
  171 & 60 & 15 & 24 & 3 & -12 & 15 & 72 \\
      & 110 & 73 & 66 & 11 & -4 & 15 & 114 \\
  \end{tabular}}
\end{minipage}\qquad
\begin{minipage}{0.4\textwidth}\centering
 {\footnotesize \begin{tabular}{cccccccc}
  $v$ & $k$ & $\lambda$ & $\mu$ & $\nu_2$ & $\nu_3$ & $\sqrt{\Lambda}$ & $m$ \\
  \hline
  190 & 36 & 18 & 4 & 16 & -2 & 18 & 38 \\
      & 153 & 120 & 136 & 1 & -17 & 18 & 170 \\
  190 & 45 & 12 & 10 & 7 & -5 & 12 & 50 \\
      & 144 & 108 & 112 & 4 & -8 & 12 & 152 \\
  190 & 84 & 33 & 40 & 4 & -11 & 15 & 95 \\
      & 105 & 60 & 55 & 10 & -5 & 15 & 110 \\
  190 & 84 & 38 & 36 & 8 & -6 & 14 & 90 \\
      & 105 & 56 & 60 & 5 & -9 & 14 & 114 \\
  190 & 90 & 45 & 40 & 10 & -5 & 15 & 95 \\
      & 99 & 48 & 55 & 4 & -11 & 15 & 110 \\
  196 & 39 & 2 & 9 & 3 & -10 & 13 & 49 \\
      & 156 & 125 & 120 & 9 & -4 & 13 & 160 \\
  196 & 60 & 23 & 16 & 11 & -4 & 15 & 64 \\
      & 135 & 90 & 99 & 3 & -12 & 15 & 147 \\
  204 & 63 & 22 & 18 & 9 & -5 & 14 & 68 \\
      & 140 & 94 & 100 & 4 & -10 & 14 & 150 \\
  210 & 38 & 19 & 4 & 17 & -2 & 19 & 40 \\
      & 171 & 136 & 153 & 1 & -18 & 19 & 189 \\
  220 & 84 & 38 & 28 & 14 & -4 & 18 & 88 \\
      & 135 & 78 & 90 & 3 & -15 & 18 & 150 \\
  222 & 51 & 20 & 9 & 14 & -3 & 17 & 54 \\
      & 170 & 127 & 140 & 2 & -15 & 17 & 185 \\
  238 & 75 & 20 & 25 & 5 & -10 & 15 & 85 \\
      & 162 & 111 & 108 & 9 & -6 & 15 & 168 \\
  244 & 108 & 42 & 52 & 4 & -14 & 18 & 121 \\
      & 135 & 78 & 70 & 13 & -5 & 18 & 140 \\
  244 & 117 & 60 & 52 & 13 & -5 & 18 & 122 \\
      & 126 & 60 & 70 & 4 & -14 & 18 & 140 \\
  246 & 85 & 20 & 34 & 3 & -17 & 20 & 102 \\
      & 160 & 108 & 96 & 16 & -4 & 20 & 164 \\
  246 & 105 & 36 & 51 & 3 & -18 & 21 & 123 \\
      & 140 & 85 & 72 & 17 & -4 & 21 & 144 \\
  246 & 119 & 64 & 51 & 17 & -4 & 21 & 123 \\
      & 126 & 57 & 72 & 3 & -18 & 21 & 144 \\
  266 & 45 & 0 & 9 & 3 & -12 & 15 & 57 \\
      & 220 & 183 & 176 & 11 & -4 & 15 & 224 \\
  273 & 80 & 19 & 25 & 5 & -11 & 16 & 91 \\
      & 192 & 136 & 132 & 10 & -6 & 16 & 198 \\
  273 & 102 & 41 & 36 & 11 & -6 & 17 & 108 \\
      & 170 & 103 & 110 & 5 & -12 & 17 & 182 \\
  273 & 136 & 65 & 70 & 6 & -11 & 17 & 147 \\
      & 136 & 69 & 66 & 10 & -7 & 17 & 143 \\
  276 & 44 & 22 & 4 & 20 & -2 & 22 & 46 \\
      & 231 & 190 & 210 & 1 & -21 & 22 & 252 \\
  276 & 75 & 10 & 24 & 3 & -17 & 20 & 92 \\
      & 200 & 148 & 136 & 16 & -4 & 20 & 204 \\
  276 & 75 & 18 & 21 & 6 & -9 & 15 & 84 \\
      & 200 & 145 & 144 & 8 & -7 & 15 & 207 \\ 
\end{tabular}}
\end{minipage}
\caption{Parameters ruled out for groups with a nontrivial center when $v \le 300$ by Corollary \ref{cor:center}}
\label{tbl:center1}
\end{table}

\begin{table}[H]
{\footnotesize \begin{tabular}{cccccccc}
  $v$ & $k$ & $\lambda$ & $\mu$ & $\nu_2$ & $\nu_3$ & $\sqrt{\Lambda}$ & $m$ \\
  \hline
  276 & 110 & 52 & 38 & 18 & -4 & 22 & 114 \\
      & 165 & 92 & 108 & 3 & -19 & 22 & 184 \\
  276 & 135 & 78 & 54 & 27 & -3 & 30 & 138 \\
      & 140 & 58 & 84 & 2 & -28 & 30 & 168 \\
  280 & 117 & 44 & 52 & 5 & -13 & 18 & 130 \\
      & 162 & 96 & 90 & 12 & -6 & 18 & 168 \\
  286 & 95 & 24 & 33 & 4 & -15 & 19 & 108 \\
      & 190 & 129 & 120 & 14 & -5 & 19 & 195 \\
  286 & 125 & 60 & 50 & 15 & -5 & 20 & 130 \\
      & 160 & 84 & 96 & 4 & -16 & 20 & 176 \\
  288 & 105 & 52 & 30 & 25 & -3 & 28 & 108 \\
      & 182 & 106 & 130 & 2 & -26 & 28 & 208 \\
  290 & 136 & 63 & 64 & 8 & -9 & 17 & 145 \\
      & 153 & 80 & 81 & 8 & -9 & 17 & 161 \\
  300 & 46 & 23 & 4 & 21 & -2 & 23 & 48 \\
      & 253 & 210 & 231 & 1 & -22 & 23 & 275 \\
\end{tabular}} 
\caption{Parameters ruled out for groups with a nontrivial center when $v \le 300$ by Corollary \ref{cor:center}, continued}
\label{tbl:center2}
\end{table}

\noindent\textsc{Acknowledgments.}  
The authors wish to thank the anonymous referees for their helpful reports.

\bibliographystyle{plain}
\bibliography{PDS}

\end{document}